\documentclass[11pt,reqno]{amsart}

\usepackage{amsfonts}
\usepackage{amsmath}
\usepackage{amssymb, esint}
\usepackage{amsthm,color}
\usepackage{enumerate}
\usepackage{mathtools}
\usepackage{enumitem}
\usepackage{url}
\usepackage[hidelinks, bookmarksdepth=3]{hyperref}
\usepackage{xcolor}
\usepackage{cancel}
\usepackage{ulem}
\usepackage{caption,subcaption}
\usepackage{tikz}
\usepackage{tikz-3dplot}
\usepackage[T1]{fontenc}
\usepackage{mathrsfs}
\usepackage{commath}

\usepackage{longtable}

\usepackage{soul}
\usepackage[
style=alphabetic,
backend=biber,
maxbibnames=99,
minbibnames=99,
maxcitenames=99,
mincitenames=99
]{biblatex}
\AtBeginDocument{
  \setlength{\abovedisplayskip}{9pt plus 2pt minus 3pt}
  \setlength{\belowdisplayskip}{9pt plus 2pt minus 3pt}
  \setlength{\abovedisplayshortskip}{6pt plus 2pt minus 2pt}
  \setlength{\belowdisplayshortskip}{6pt plus 2pt minus 2pt}
  \setlength{\jot}{4pt}
}

\DeclareFieldFormat[article]{journaltitle}{#1}
\DeclareFieldFormat[book]{title}{#1}

\tdplotsetmaincoords{70}{110}

\tdplotsetmaincoords{70}{110}

\usepackage[a4paper,margin=2.8cm]{geometry}

\def\0{{\mathbf 0}}

\theoremstyle{plain}
\newtheorem{thm}{Theorem}[section]

\newtheorem{prop}[thm]{Proposition}

\newcommand{\thistheoremnames}{}
\newtheorem*{genericthms}{\thistheoremnames}
\newenvironment{para*}[1]
  {\renewcommand{\thistheoremnames}{#1}%
   \begin{genericthms}}
  {\end{genericthms}}

\theoremstyle{remark}

\newtheorem*{claim*}{Claim}

\newtheorem{rem}[thm]{Remark}

\begin{document}

\numberwithin{equation}{section}

\title[A strengthening of Chang's lemma]{A strengthening of Chang's lemma}

\author[G. Carenini]{Gaia Carenini}
\address{Trinity College Cambridge, Department of Pure Mathematics and Mathematical Statistics, Centre for Mathematical Sciences, Wilberforce Road, Cambridge CB3 0WA, United Kingdom.}
\email{gc645@cam.ac.uk}

\author[L. Franchi]{Leonardo Franchi}
\address{Department of Pure Mathematics and Mathematical Statistics, Centre for Mathematical Sciences, Wilberforce Road, Cambridge CB3 0WA, United Kingdom}
\email{lf511@cam.ac.uk}

\begin{abstract}
We prove a strengthening of Chang's lemma for subsets of $\mathbb F_p^n$. The
classical conclusion that the large spectrum is contained in a subspace of
dimension at most $2\varepsilon^{-2}\log(1/\alpha)$ is refined to show that
every character outside this subspace has small correlation with the set not
only globally, but also on average over the cosets of the orthogonal complement,
in a natural cosetwise $\ell^1$ norm. As a consequence, we obtain a localized
counting lemma. We also give an extension of the argument to
arbitrary finite abelian groups.
\end{abstract}
\maketitle
\vspace{-0.8cm}
\tableofcontents
\section{Introduction}
A cornerstone of additive combinatorics, Chang’s lemma \cite{Chang2002} asserts that the large Fourier spectrum of a sparse set is highly structured, being contained in a low-dimensional subspace. Initially motivated by applications to improving Freiman’s theorem on sumsets, this lemma has subsequently found numerous applications across several areas of mathematics.  Most notably, it has been extensively used within additive combinatorics and combinatorial number theory. For instance, it plays a key role in establishing structural results about arithmetic progressions in sumsets \cite{green2008quantitative}, in the analysis of Boolean functions with small spectral norm \cite{green2008boolean}, in obtaining improved bounds for Roth’s theorem on three-term arithmetic progressions in the integers \cite{sanders2011roth, bloom2016quantitative, bloom2020breaking}, and in strengthening the upper bound for the size of sets with no square difference \cite{bloom2022new}. It is also closely connected to Fourier-analytic approaches to additive problems such as Littlewood–Offord-type inequalities \cite{sanders2007littlewood}. Beyond additive combinatorics, Chang’s lemma has had significant impact in other areas. It plays a central role in combinatorial group theory, particularly in the study of growth in groups and approximate subgroups \cite{Breuillard2012, helfgott2015growth}. Furthermore, it has found important applications in theoretical computer science and learning theory, including Fourier-based learning algorithms and structural approximation results \cite{tsang2013fourier, chan2016approximate, lee2019fourier, asadi2022worst}.

As is common in additive combinatorics, we first work in the setting of vector spaces over a fixed finite field. This setting contains the main ideas while keeping the notation transparent. The corresponding statement for arbitrary finite abelian groups
requires a slightly different formulation and is postponed to Section 4.

Let $p$ be a prime. For any two elements $\xi,x$ of $\mathbb{F}_p^n$, we define their inner product as $\langle \xi,x\rangle = \sum_{i=1}^n \xi_i x_i \in \mathbb{F}_p$. To each element $\xi\in\mathbb{F}_p^n$, we associate the character $\chi_\xi:\mathbb{F}_p^n\to\mathbb{C}$ given by $\chi_\xi(x)=\omega^{\langle \xi,x\rangle}$, where $\omega=e^{2\pi i/p}$. The family $\{\chi_\xi:\xi\in\mathbb{F}_p^n\}$ forms an orthonormal basis of $L^2(\mathbb{F}_p^n)$ with respect to the inner product $$\langle f,h\rangle=\mathop{\mathbb E}\limits_{x\in\mathbb{F}_p^n}f(x)\overline{h(x)},$$ where  $\mathop{\mathbb E}\limits_{x\in\mathbb{F}_p^n}$ denotes the expectation with respect to the uniform measure on $\mathbb{F}_p^n$.

For every $\xi\in\mathbb{F}_p^n$, we define the Fourier transform of a function $f:\mathbb{F}_p^n\to\mathbb{C}$ at $\xi$ by
$$
\widehat{f}(\xi)=\mathop{\mathbb E}\limits_{x\in\mathbb{F}_p^n} f(x)\overline{\chi_\xi(x)}.
$$

For every set $A\subseteq \mathbb{F}_p^n$, we denote by $\mathbf{1}_A$ its indicator function.
Given a set $A\subseteq \mathbb{F}_p^n$ of density $\alpha$ and a real parameter $\varepsilon>0$, we define the large spectrum of $A$ by
$$
\operatorname{Spec}_\varepsilon(A):=
\left\{
\xi\in \mathbb{F}_p^n:
\left|\widehat{\mathbf{1}_A}(\xi)\right|> \varepsilon\alpha
\right\}.
$$
With this notation in place, we can now state Chang's lemma in the setting of $\mathbb{F}_p^n$.

\begin{thm}[Chang's lemma]\label{theo:Chang-original}
Let $A\subseteq \mathbb{F}_p^n$ be a set of density $\alpha$, and let $\varepsilon\in(0,1)$.
Then there exists a subspace $V\leq \mathbb{F}_p^n$ with
\[
\dim(V)\leq \frac{2}{\varepsilon^2}\log\frac{1}{\alpha}
\]
such that $\operatorname{Spec}_\varepsilon(A)\subseteq V$.
\end{thm}
To put this in perspective, Parseval's identity yields the cardinality bound
$
|\operatorname{Spec}_\varepsilon(A)| \le \varepsilon^{-2}\alpha^{-1}.
$
This shows that the large spectrum cannot be too large, but it does not rule out the possibility that the large Fourier
coefficients are spread in an unstructured way across the dual group. Chang's
lemma gives a stronger structural conclusion, replacing this cardinality bound with the
statement that the large spectrum is contained in a subspace of dimension at
most $2\varepsilon^{-2}\log(1/\alpha)$.

\subsection{Strengthened Chang's lemma}
It is natural to ask whether Theorem~\ref{theo:Chang-original} can be sharpened. Broadly speaking, there are two directions for improvement: one could seek better quantitative bounds, or one could try to strengthen the structural content of its conclusion. The former direction is essentially ruled out by a result of Green \cite{Green2003}, who showed that Chang’s lemma is quantitatively optimal in general, establishing that the dimension bound is sharp up to constant factors. Indeed, in some regimes, this is already witnessed over $\mathbb{F}_2^n$ by Hamming balls. Consequently, the only viable option is the second, and this is the direction we explore in this work. 

Before stating our result, we present an alternative formulation of Chang’s lemma that will make the comparison with our result more transparent. Given a subspace $V\le \mathbb{F}_p^n$, write
$$
V^\perp:=\{x\in\mathbb{F}_p^n:\langle x,\xi\rangle=0 \text{ for every }\xi\in V\},
$$
and let
$
\pi_V:\mathbb{F}_p^n\to \mathbb{F}_p^n/V^\perp
$
denote the quotient map. For each $y\in \mathbb F_p^n/V^\perp$, we identify $y$
with the corresponding coset $y+V^\perp$ of $V^\perp$ in $\mathbb F_p^n$. 
\begin{thm}[Chang's lemma, coset form]
\label{theo:chang-reformulated}
Let $A\subseteq \mathbb{F}_p^n$ be a set of density $\alpha$, and let $\varepsilon\in(0,1)$.
Then there exists a subspace $V\leq \mathbb{F}_p^n$ with
$$
\dim(V)\leq \frac{2}{\varepsilon^2}\log\frac{1}{\alpha}
$$
such that
\begin{equation}\label{eq:chang-reformulated}
\left|
\mathop{\mathbb E}\limits_{y\in \mathbb F_p^n/V^\perp}
\mathop{\mathbb E}\limits_{x\in y+V^\perp}
\mathbf{1}_A(x)\,\overline{\chi_\xi(x)}
\right|
\le \varepsilon\alpha
\end{equation}
for every $\xi\notin V$.
\end{thm}
The equivalence with the classical Chang's lemma (Theorem~\ref{theo:Chang-original}) holds since the left-hand side of equation \eqref{eq:chang-reformulated} is exactly  $|\widehat{\mathbf 1_A}(\xi)|$. Thus, every Fourier coefficient outside $V$ has a magnitude of at most $\varepsilon\alpha$, and hence $\operatorname{Spec}_\varepsilon(A)\subseteq V$. Our main result is the following strengthening of Theorem~\ref{theo:chang-reformulated}.

\begin{thm}[Strengthened Chang's lemma]\label{theo:main}
Let $A\subseteq \mathbb{F}_p^n$ be a set of density $\alpha$, and let $\varepsilon\in(0,1)$.
Then there exists a subspace $V\le \mathbb{F}_p^n$ with
$$
\dim(V)\le \frac{2}{\varepsilon^2}\log\frac{1}{\alpha}
$$
such that
\begin{equation}\label{eq:refined-chang}
\mathop{\mathbb E}\limits_{y\in \mathbb F_p^n/V^\perp}
\left|
\mathop{\mathbb E}\limits_{x\in y+V^\perp}
\mathbf{1}_A(x)\,\overline{\chi_\xi(x)}
\right|
\le \varepsilon\alpha
\end{equation}
for every $\xi\notin V$.
\end{thm}
Roughly speaking, Theorem~\ref{theo:main} shows that the ``smallness'' of the Fourier coefficients outside $V$ is not merely a consequence of global cancellation, but is already witnessed by substantial cancellation on many fibers. 

A closer inspection of inequality \eqref{eq:refined-chang} for a fixed frequency $\xi$ shows that the conclusion of Theorem~\ref{theo:main} becomes stronger as the dimension of $V$ increases: indeed, larger subspaces correspond to smaller cosets of $V^\perp$, and therefore leave less room for cancellation within each coset. In particular, any argument showing that condition \eqref{eq:refined-chang} fails for sufficiently large values of $\dim(V)$ immediately yields an improvement on the classical dimension bound in Chang's lemma.

\subsection{Localized counting}
To illustrate the kind of benefit that Theorem \ref{theo:main} offers, we present a new counting result. While the same argument applies to counts associated with any fixed linear equation over $\mathbb F_p^n$, we state the result for additive energy, which is the simplest example.
For a function $f:\mathbb F_p^n\to\mathbb C$, we define the following quantity
$$
\Lambda_4(f)
=
\mathop{\mathbb E}\limits_{x_1+x_2=x_3+x_4}
f(x_1)f(x_2)\overline{f(x_3)f(x_4)}.
$$
When $f=\mathbf 1_A$ for a set $A\subseteq \mathbb F_p^n$, this is, up to a constant factor, the additive energy of $A$
$$
\Lambda_4(\mathbf 1_A)
=
\frac{|\{(x,y,z,w)\in A^4:x+y=z+w\}|}{|\mathbb F_p^n|^3}.
$$
The following counting lemma for additive energy is a standard folklore consequence of Chang's lemma.
\begin{prop}\label{prop:counting-lemma}
Let $A\subseteq \mathbb F_p^n$ have density $\alpha$, and let $\varepsilon\in(0,1)$. Then there exists a subspace $W\le \mathbb F_p^n$ with $\operatorname{codim}(W)\le \frac{2}{\varepsilon^2}\log\frac{1}{\alpha}$
such that, if $g(x)=\mathop{\mathbb E}\limits_{y\in x+W}\mathbf 1_A(y)$,
then
$$
|\Lambda_4(\mathbf 1_A)-\Lambda_4(g)|\le \varepsilon^2\alpha^3.
$$
\end{prop}
The proof of this proposition immediately follows by combining Chang's lemma with the Fourier identity
$$
\Lambda_4(f)=\sum_{\xi\in\mathbb F_p^n} |\widehat f(\xi)|^4.
$$

Note that Proposition~\ref{prop:counting-lemma} is still a global statement: it shows that the total additive energy of $A$ is well approximated by that of its fibre-density model $g$, but it does not say how this additive structure is distributed among the fibres. In particular, global additive energy tells you how much additive structure there is, but not where that structure is located.

One would therefore like to understand whether the stronger fiberwise cancellation in Theorem~\ref{theo:main} allows one to localise this approximation. Since Theorem~\ref{theo:main} controls the oscillatory part of $\mathbf 1_A$ on average over the fibres, it is natural to ask whether one may insert bounded weights that are constant on each fibre and still approximate the resulting weighted count by the simpler fibre-density model $g$. The next proposition shows that this is indeed possible.

To formulate this precisely, we introduce the following multilinear form
$$
\Lambda_4(f_1,f_2,f_3,f_4)
=
\mathop{\mathbb E}\limits_{x_1+x_2=x_3+x_4}
f_1(x_1)f_2(x_2)\overline{f_3(x_3)f_4(x_4)}.
$$
By definition, we have  $\Lambda_4(f)=\Lambda_4(f,f,f,f)$ and the standard identity
$$
\Lambda_4(f_1,f_2,f_3,f_4)
=
\sum_{\xi\in\mathbb F_p^n}
\widehat f_1(\xi)\widehat f_2(\xi)\,
\overline{\widehat f_3(\xi)\widehat f_4(\xi)}.
$$
We can now state the localized counting lemma.
\begin{prop}[Localized counting]\label{prop:localized-counting}
Let $A\subseteq \mathbb F_p^n$ have density $\alpha$, and let $\varepsilon\in(0,1)$. Then there exists a subspace $W\le \mathbb F_p^n$ with $\operatorname{codim}(W)\le \frac{2}{\varepsilon^2}\log\frac{1}{\alpha}$ such that, if $g(x)=\mathop{\mathbb E}\limits_{y\in x+W}\mathbf 1_A(y)$, then for every choice of functions $q_1,q_2,q_3,q_4:\mathbb F_p^n\to\mathbb C$ that are constant on the cosets of $W$ and satisfy $|q_i|\le 1$, one has the uniform estimate
$$
\left|
\Lambda_4(q_1\mathbf 1_A,q_2\mathbf 1_A,q_3\mathbf 1_A,q_4\mathbf 1_A)
-
\Lambda_4(q_1g,q_2g,q_3g,q_4g)
\right|
\le \varepsilon^2\alpha^3.
$$
\end{prop}
The value of Proposition~\ref{prop:localized-counting} lies in the uniformity over the weights. Indeed, if one fixes a quadruple $q_1,q_2,q_3,q_4$ and repeats the Fourier proof of Proposition~\ref{prop:counting-lemma} using only the classical Chang's lemma, one obtains the estimate
$$
\left|
\Lambda_4(q_1\mathbf 1_A,q_2\mathbf 1_A,q_3\mathbf 1_A,q_4\mathbf 1_A)
-
\Lambda_4(q_1g,q_2g,q_3g,q_4g)
\right|
\le
\varepsilon^2\alpha^3
\min_{1\le i<j\le 4}
\|\widehat q_i\|_{\ell^1}\|\widehat q_j\|_{\ell^1}.
$$
Thus, the classical argument gives a bound for each fixed choice of weights, but with a constant depending on their Fourier $\ell^1$ norms. Proposition~\ref{prop:localized-counting} removes this dependence.

\subsection{Structure of the paper}

The remainder of the paper is organized as follows. In Section~2, we prove the strengthened Chang's lemma in $\mathbb F_p^n$.  In Section~3 we prove the
localized counting lemma for additive energy. Finally, in Section~4, we discuss the extension to
arbitrary finite abelian groups. 

\medskip

\textbf{Acknowledgments.}
We are grateful to our supervisors, Imre Leader and Timothy Gowers, for their support and guidance; the first author is supported by the CB European PhD Studentship funded by Trinity College, Cambridge, and the second author acknowledges support from the Isaac Newton Trust through a Trinity Cambridge Research Studentship.
We also thank Julia Wolf for stimulating discussions and for pointing out the reference \cite{Lee2017LargeSpectrum}.

\section{Proof of the strengthened Chang's lemma}
Before we turn to the proof of Theorem \ref{theo:main}, let us say a few words about the existing proofs of Chang’s lemma.
The original proof by Chang relies on Rudin's inequality \cite{Chang2002}. Later proofs proceed by proving the level-$1$ inequality on the Boolean cube via entropy, as in \cite{impagliazzo2014entropic}, and by extending this argument to $\mathbb{F}_p^n$, as in \cite{WigdersonAlmostPeriodicity}. A different approach was later developed by Lee \cite{Lee2017LargeSpectrum}, based on entropy maximization and generalized Riesz products. None of these arguments seems to adapt in a straightforward way to the present setting. 

Our proof is based on a natural iterative construction in which relative entropy appears only as a progress measure. At each step, if the current subspace fails to satisfy the conclusion of the theorem, we choose a frequency witnessing this failure and enlarge the subspace by adjoining it. A peculiarity of the proof is that we keep track throughout the iteration of the behaviour of characters on individual cosets. In this respect, the argument is inspired by the way boosting procedures are adapted to satisfy group-fairness constraints, particularly in the context of multicalibration \cite{HebertJohnsonKimReingoldRothblum2018}.

 For any non-empty subset $S \subseteq \mathbb{F}_p^n$, let $\mu_S$ denote the uniform probability measure on $S$, defined by $\mu_S(x)=\mathbf{1}_S(x)/|S|$ for all $x \in \mathbb{F}_p^n$. Let $A$ be the set appearing in the statement of Theorem~\ref{theo:main}, and let $\alpha=|A|/p^n$ be its density. We write $\mu=\mu_A$. For any subset $E\subseteq \mathbb{F}_p^n$, we write $\mu(E)=\sum_{x\in E}\mu(x)$. If $F\subseteq \mathbb{F}_p^n$ is an affine subspace and $\mu(F)>0$, we define
$$
\mathbb{E}_{x\sim\mu\mid F} f(x)
=
\frac{1}{\mu(F)}\sum_{x\in F}\mu(x)f(x),
$$
and, by convention, we interpret this quantity as $0$ when $\mu(F)=0$.

For each subspace $V\leq \mathbb{F}_p^n$ and each $y\in \mathbb{F}_p^n/V^\perp$, we write $F_{V,y}=y+V^\perp$ for the corresponding fiber. Then, for every subspace $V\leq \mathbb{F}_p^n$, every $\xi\in\mathbb{F}_p^n$, and every $y\in \mathbb{F}_p^n/V^\perp$, one has
$$
\mu(F_{V,y})\,\mathbb E_{x\sim\mu\mid F_{V,y}} \overline{\chi_\xi(x)}
=
\sum_{x\in F_{V,y}}\mu(x)\overline{\chi_\xi(x)}
=
\frac{1}{\alpha p^n}\sum_{x\in F_{V,y}}\mathbf{1}_A(x)\overline{\chi_\xi(x)}.
$$
Using this identity, we restate Theorem~\ref{theo:main} in a form that is more convenient for the proof.

\begin{thm}\label{theo:main-restated}
Let $A\subseteq \mathbb{F}_p^n$ be a set of density $\alpha>0$, and let $\varepsilon\in(0,1)$. Then there exists a subspace $V\leq \mathbb{F}_p^n$ with
$
\dim(V)\leq \frac{2}{\varepsilon^2}\log\frac{1}{\alpha}
$
such that
\begin{equation}\label{eq:main-restated}
\sum_{y\in \mathbb F_p^n/V^\perp}
\mu(F_{V,y})
\left|
\mathbb E_{x\sim\mu\mid F_{V,y}} \overline{\chi_\xi(x)}
\right|
\leq \varepsilon
\end{equation}
for every $\xi\notin V$.
\end{thm}
Before presenting the proof, let us finally recall that, given two probability measures $\nu_1$ and $\nu_2$ on $\mathbb{F}_p^n$, the Kullback--Leibler divergence of $\nu_1$ from $\nu_2$ is defined by
$$
D(\nu_1\|\nu_2):=
\sum_{x\in \mathbb{F}_p^n} \nu_1(x)\log\frac{\nu_1(x)}{\nu_2(x)},
$$
with the usual conventions that $0\log 0=0$ and that $D(\nu_1\|\nu_2)=+\infty$ if there exists $x\in \mathbb{F}_p^n$ such that $\nu_1(x)>0$ and $\nu_2(x)=0$.
\begin{proof}[Proof of Theorem \ref{theo:main-restated}]
We construct inductively a nested sequence of subspaces
$$
V_0 \leq V_1 \leq \cdots \leq V_T
$$
of $\mathbb{F}_p^n$, and define for every $0\leq i\leq T$ probability measures $\nu_i$ on $\mathbb{F}_p^n$, where each $\nu_i$ is constant on every coset of $V_i^\perp$. We let $V_0=\{0\}$ and $\nu_0=\mu_{\mathbb{F}_{p}^n}$ and assume that the subspace $V_i$ and the measure $\nu_i$ have already been defined. If for every $\xi\notin V_i$, the inequality \eqref{eq:main-restated} is satisfied, then the process stops. Otherwise, choose a vector $\xi\notin V_i$ witnessing the failure of the conclusion, that is,
$$
\sum_{y\in \mathbb F_p^n/V_i^\perp}
\mu(F_{V_i,y})
\left|
\mathop{\mathbb E}\limits_{x\sim\mu\mid F_{V_i,y}} \overline{\chi_\xi(x)}
\right|
>\varepsilon.
$$
For each $y\in \mathbb F_p^n/V_i^\perp$, we choose a phase factor $\phi(y)$ so as to align
${\mathbb E}_{x\sim\mu\mid F_{V_i,y}} \overline{\chi_\xi(x)}$ with the positive real axis. If this quantity is non-zero, we define $\phi(y)$ to be its complex conjugate normalized to have modulus $1$; otherwise we let $\phi(y)=1$. We use this phase function to define a real-valued test function $d:\mathbb{F}_p^n\to[-1,1]$ by
$$
d(x)=\Re\bigl(\phi(y)\,\overline{\chi_\xi(x)}\bigr)
\qquad\text{whenever }x\in F_{V_i,y}.
$$
With this definition, the expectation of $d$ under $\mu$ is
$$
\mathop{\mathbb E}\limits_{x\sim\mu}d(x)
=
\sum_{y\in \mathbb F_p^n/V_i^\perp}
\mu(F_{V_i,y})\,
\Re\left(
\phi(y)\,
\mathop{\mathbb E}\limits_{x\sim\mu\mid F_{V_i,y}} \overline{\chi_\xi(x)}
\right)
=
\sum_{y\in \mathbb F_p^n/V_i^\perp}
\mu(F_{V_i,y})
\left|
\mathop{\mathbb E}\limits_{x\sim\mu\mid F_{V_i,y}} \overline{\chi_\xi(x)}
\right|
>
\varepsilon.
$$

Set $V_{i+1}=\operatorname{span}(V_i,\xi)$. Observe that, since $\chi_\xi$ is constant on cosets of $\xi^\perp$ and $\phi(y)$ depends only on the coset $F_{V_i,y}$, the function $d$ is constant on cosets of $V_i^\perp\cap \xi^\perp = V_{i+1}^\perp$. Hence $d$ is constant on every coset of $V_{i+1}^\perp$. We now define
$$
\nu_{i+1}(x)=(1+\varepsilon d(x))\nu_i(x).
$$
By construction, $\nu_{i+1}$ is constant on every coset of $V_{i+1}^\perp$. Moreover, since $d(x)\in[-1,1]$ and $\varepsilon\in(0,1)$, the factor $1+\varepsilon d(x)$ is strictly positive for every $x\in\mathbb F_p^n$.
We next check that $\nu_{i+1}$ is still a probability measure. Since $\xi\notin V_i$, the character $\chi_\xi$ is non-trivial on $V_i^\perp$, and therefore its average on every coset of $V_i^\perp$ vanishes. As $\nu_i$ is constant on these cosets, it follows that
$$
\mathbb E_{x\sim\nu_i} d(x)
=
\sum_{y\in \mathbb F_p^n/V_i^\perp}
\nu_i(F_{V_i,y})\,
\Re\left(
\phi(y)\,
\mathop{\mathbb E}\limits_{x\in F_{V_i,y}} \overline{\chi_\xi(x)}
\right)
=
0.
$$
Hence $\sum_x \nu_{i+1}(x)=\mathbb E_{x\sim\nu_i}(1+\varepsilon d(x))=1$, so $\nu_{i+1}$ is a probability measure.
Defining the potential $\Psi_i=D(\mu\|\nu_i)$, we then have
$$
\Psi_{i+1}
=
\sum_{x\in\mathbb F_p^n}
\mu(x)\log\frac{\mu(x)}{\nu_{i+1}(x)}
=
\Psi_i-\mathbb E_{x\sim\mu}\log(1+\varepsilon d(x)).
$$
We now lower bound the last expectation. Since $t\mapsto\log(1+\varepsilon t)$ is concave on $[-1,1]$, its graph lies above the chord joining the endpoints. Thus, for every $t\in[-1,1]$,
$$
\log(1+\varepsilon t)
\geq
\frac{1+t}{2}\log(1+\varepsilon)
+
\frac{1-t}{2}\log(1-\varepsilon).
$$
Averaging with respect to $\mu$ gives
$$
\mathbb E_{x\sim\mu}\log(1+\varepsilon d(x))
\geq
\frac{1+\mathbb E_{x\sim\mu}d(x)}{2}\log(1+\varepsilon)
+
\frac{1-\mathbb E_{x\sim\mu}d(x)}{2}\log(1-\varepsilon).
$$
The right-hand side is increasing as a function of $\mathbb E_{x\sim\mu}d(x)$. Since $\mathbb E_{x\sim\mu}d(x)>\varepsilon$, we obtain
$$
\mathbb E_{x\sim\mu}\log(1+\varepsilon d(x))
\geq
\frac{1+\varepsilon}{2}\log(1+\varepsilon)
+
\frac{1-\varepsilon}{2}\log(1-\varepsilon).
$$
Finally, the function
$$
h(t)=\frac{1+t}{2}\log(1+t)+\frac{1-t}{2}\log(1-t)
$$
satisfies $h(0)=0$ and $h'(t)=\frac12\log\frac{1+t}{1-t}\geq t$ for every $t\in[0,1)$. Therefore $h(\varepsilon)\geq \varepsilon^2/2$, and hence
$$
\Psi_{i+1}\leq \Psi_i-\frac{\varepsilon^2}{2}.
$$
In the end, since $\nu_0=\mu_{\mathbb F_p^n}$ and $\mu=\mu_A$, we have $\Psi_0=D(\mu_A\|\mu_{\mathbb F_p^n})=\log(1/\alpha)$. On the other hand, $\Psi_i\geq0$ for every $i$. Therefore the process must terminate after at most
$$
T\leq \frac{2}{\varepsilon^2}\log\frac{1}{\alpha}
$$
steps, and identifying $V$ with $V_T$ concludes the proof.
\end{proof}
\begin{rem}
The proof above can be viewed as part of a more general iterative argument.
If one inspects it carefully, the specific use of a Fourier character is not
essential. What matters is that, whenever the current quotient does not yet
capture the relevant structure, one can find a bounded test function which has
mean zero on each fibre of the current quotient, is constant on the fibres of a
linear refinement of controlled rank, and has positive average against the
target density. Once such a function is available, the proof scheme works exactly in the same way.

For each subspace $V\le \mathbb F_p^n$ and each function
$h:\mathbb F_p^n\to[0,\infty)$, define
$$
h_V(x)=\mathop{\mathbb E}\limits_{y\in x+V^\perp} h(y).
$$
Thus $h_V$ is constant on the cosets of $V^\perp$. The proof of Theorem \ref{theo:main} can be adapted to
show that, if
$$
\mathop{\mathbb E}\limits_{x\in\mathbb F_p^n} h(x)=1
\qquad\text{and}\qquad
\mathop{\mathbb E}\limits_{x\in\mathbb F_p^n}\bigl[h(x)\log h(x)\bigr]<\infty,
$$
then, for every $\varepsilon\in(0,1)$ and every integer $r\ge 1$, there exists
a subspace $V\le \mathbb F_p^n$ with
$$
\dim(V)\le \frac{2r}{\varepsilon^2}\,
\mathop{\mathbb E}\limits_{x\in\mathbb F_p^n}\bigl[h(x)\log h(x)\bigr]
$$
such that, for every subspace $V'\ge V$ satisfying
$\dim(V')-\dim(V)\le r$, one has
$$
\mathop{\mathbb E}\limits_{x\in\mathbb F_p^n}|h_{V'}(x)-h_V(x)|\le \varepsilon.
$$

One can also see that Theorem~\ref{theo:main-restated} can be recovered from
this statement by choosing $h=\alpha^{-1}\mathbf 1_A$ and $r=1$, and then, for
each $\xi\notin V$, by taking $V'=\operatorname{span}(V,\xi)$. In this case,
$h_{V'}-h_V$ captures the additional oscillation revealed by refining the
quotient in the direction of $\xi$. Indeed, $\chi_\xi$ is constant on the
cosets of $V'^\perp$ and has mean zero on each coset of $V^\perp$, so the
fibrewise quantity in \eqref{eq:main-restated} is controlled by
$\mathop{\mathbb E}\limits_{x\in\mathbb F_p^n}|h_{V'}(x)-h_V(x)|$.
\end{rem}
\section{Proof of the localized counting lemma}

We now turn our attention to the proof of Proposition~\ref{prop:localized-counting}. We identify functions on $\mathbb F_p^n/W$ with functions on $\mathbb F_p^n$ that are constant on the cosets of $W$. When working on the subgroup $W\le \mathbb F_p^n$, we write $\widehat W$ for its dual group, that is, the group of characters $\gamma:W\to\mathbb C^\times$. Using the standard pairing on $\mathbb F_p^n$, we identify $\widehat W$ with $\mathbb F_p^n/W^\perp$ through the map $\xi\mapsto \chi_\xi|_W$.
\begin{proof}
By Theorem~\ref{theo:main}, there exists a subspace $V\le \mathbb F_p^n$ with
$
\dim(V)\le \frac{2}{\varepsilon^2}\log\frac{1}{\alpha}
$
such that, setting $W=V^\perp$, one has
$$
\mathop{\mathbb E}\limits_{z\in \mathbb F_p^n/W}
\left|
\mathop{\mathbb E}\limits_{x\in z+W}\mathbf 1_A(x)\overline{\chi_\xi(x)}
\right|
\le \varepsilon\alpha
\qquad\text{for every }\xi\notin W^\perp.
$$
In particular,
$
\operatorname{codim}(W)=\dim(V)\le \frac{2}{\varepsilon^2}\log\frac{1}{\alpha}.
$
We also regard $g$ as a function on $\mathbb F_p^n/W$, so that $g(z)$ denotes the common value of $g$ on the coset $z$. Set
$$
\Delta=\Lambda_4(q_1\mathbf 1_A,q_2\mathbf 1_A,q_3\mathbf 1_A,q_4\mathbf 1_A)-\Lambda_4(q_1g,q_2g,q_3g,q_4g).
$$
Since the functions $q_i$ and $g$ are constant on the cosets of $W$, averaging first over the cosets containing the variables gives the following expression for $\Delta$
\begin{small}
$$
\mathop{\mathbb E}\limits_{z_1+z_2=z_3+z_4}
q_1(z_1)q_2(z_2)\overline{q_3(z_3)q_4(z_4)}
\left(
\mathop{\mathbb E}\limits_{\substack{x_i\in z_i\\ x_1+x_2=x_3+x_4}}
\mathbf 1_A(x_1)\mathbf 1_A(x_2)\mathbf 1_A(x_3)\mathbf 1_A(x_4)
-
g(z_1)g(z_2)g(z_3)g(z_4)
\right).
$$
\end{small}
Now, fix cosets $z_1,z_2,z_3,z_4$ with
$
z_1+z_2=z_3+z_4.
$
Choose points $x_i\in z_i$ for $i=1,2,3,4$ such that $x_1+x_2=x_3+x_4$, and define $h_i:W\to\mathbb C$ by $h_i(w)=\mathbf 1_A(x_i+w)$.
Then
$$
\mathop{\mathbb E}\limits_{\substack{x_i\in z_i\\ x_1+x_2=x_3+x_4}}
\mathbf 1_A(x_1)\mathbf 1_A(x_2)\mathbf 1_A(x_3)\mathbf 1_A(x_4)
=
\mathop{\mathbb E}\limits_{w_1+w_2=w_3+w_4}
h_1(w_1)h_2(w_2)h_3(w_3)h_4(w_4).
$$
Hence, by the Fourier identity on $W$,
$$
\mathop{\mathbb E}\limits_{\substack{x_i\in z_i\\ x_1+x_2=x_3+x_4}}
\mathbf 1_A(x_1)\mathbf 1_A(x_2)\mathbf 1_A(x_3)\mathbf 1_A(x_4)
=
\sum_{\gamma\in\widehat W}
\widehat h_1(\gamma)\widehat h_2(\gamma)\,
\overline{\widehat h_3(\gamma)\widehat h_4(\gamma)}.
$$
For each character $\gamma\in\widehat W$ and each coset $z\in\mathbb F_p^n/W$, define
$$
a_\gamma(z)
=
\left|
\mathop{\mathbb E}\limits_{x\in z}\mathbf 1_A(x)\overline{\chi_\xi(x)}
\right|,
$$
where $\xi\in\mathbb F_p^n$ is any element such that $\chi_\xi|_W=\gamma$. Note that this is a well-defined quantity. If $\chi_\xi|_W=\gamma$, then for each $i$ we have
$$
\mathop{\mathbb E}\limits_{x\in z_i}\mathbf 1_A(x)\overline{\chi_\xi(x)}
=
\overline{\chi_\xi(x_i)}\,
\mathop{\mathbb E}\limits_{w\in W}\mathbf 1_A(x_i+w)\overline{\gamma(w)}
=
\overline{\chi_\xi(x_i)}\,\widehat h_i(\gamma),
$$
and therefore
$
|\widehat h_i(\gamma)|=a_\gamma(z_i).
$
Moreover, when $\gamma$ is the trivial character, one has
$
a_\gamma(z_i)=g(z_i).
$
Substituting this into the previous Fourier expansion, and using $|q_i|\le 1$, we obtain
$$
|\Delta|
\le
\sum_{\substack{\gamma\in\widehat W\\ \gamma\text{ non-trivial}}}
\mathop{\mathbb E}\limits_{z_1+z_2=z_3+z_4}
a_\gamma(z_1)a_\gamma(z_2)a_\gamma(z_3)a_\gamma(z_4).
$$
By the Fourier identity on $\mathbb F_p^n/W$,
$$
\mathop{\mathbb E}\limits_{z_1+z_2=z_3+z_4}
a_\gamma(z_1)a_\gamma(z_2)a_\gamma(z_3)a_\gamma(z_4)
=
\sum_{\eta\in\widehat{\mathbb F_p^n/W}} |\widehat{a_\gamma}(\eta)|^4.
$$
Hence,
$$
\mathop{\mathbb E}\limits_{z_1+z_2=z_3+z_4}
a_\gamma(z_1)a_\gamma(z_2)a_\gamma(z_3)a_\gamma(z_4)
\le
\left(\sup_{\eta\in\widehat{\mathbb F_p^n/W}} |\widehat{a_\gamma}(\eta)|^2\right)
\sum_{\eta\in\widehat{\mathbb F_p^n/W}} |\widehat{a_\gamma}(\eta)|^2.
$$
Since $a_\gamma\ge 0$, we have
$
|\widehat{a_\gamma}(\eta)|\le
\mathop{\mathbb E}\limits_{z\in\mathbb F_p^n/W} a_\gamma(z),
$
so Parseval gives
$$
\mathop{\mathbb E}\limits_{z_1+z_2=z_3+z_4}
a_\gamma(z_1)a_\gamma(z_2)a_\gamma(z_3)a_\gamma(z_4)
\le
\left(
\mathop{\mathbb E}\limits_{z\in\mathbb F_p^n/W} a_\gamma(z)
\right)^2
\mathop{\mathbb E}\limits_{z\in\mathbb F_p^n/W} a_\gamma(z)^2.
$$
Therefore,
$$
|\Delta|
\le
\sum_{\substack{\gamma\in\widehat W\\ \gamma\text{ non-trivial}}}
\left(
\mathop{\mathbb E}\limits_{z\in\mathbb F_p^n/W} a_\gamma(z)
\right)^2
\mathop{\mathbb E}\limits_{z\in\mathbb F_p^n/W} a_\gamma(z)^2.
$$
Now, fix a non-trivial character $\gamma\in\widehat W$, and choose $\xi\in\mathbb F_p^n$ such that $\chi_\xi|_W=\gamma$. Since $\gamma$ is non-trivial, one has $\xi\notin W^\perp$. Hence, by the choice of $W$,
$$
\mathop{\mathbb E}\limits_{z\in\mathbb F_p^n/W} a_\gamma(z)
=
\mathop{\mathbb E}\limits_{z\in\mathbb F_p^n/W}
\left|
\mathop{\mathbb E}\limits_{x\in z}\mathbf 1_A(x)\overline{\chi_\xi(x)}
\right|
\le
\varepsilon\alpha.
$$
Finally, fix $z\in\mathbb F_p^n/W$, choose any $x_0\in z$, and define $h:W\to\mathbb C$ by $h(w)=\mathbf 1_A(x_0+w)$. Then Parseval on $W$ gives
$$
\sum_{\gamma\in\widehat W} |\widehat h(\gamma)|^2
=
\mathop{\mathbb E}\limits_{w\in W}|h(w)|^2
=
\mathop{\mathbb E}\limits_{x\in z}\mathbf 1_A(x)
=
g(z).
$$
By the definition of $a_\gamma(z)$, this means
$
\sum_{\gamma\in\widehat W} a_\gamma(z)^2=g(z).
$
Averaging over $z\in\mathbb F_p^n/W$, we get
$$
\sum_{\gamma\in\widehat W}
\mathop{\mathbb E}\limits_{z\in\mathbb F_p^n/W} a_\gamma(z)^2
=
\mathop{\mathbb E}\limits_{z\in\mathbb F_p^n/W} g(z)
=
\alpha.
$$
Therefore,
$$
|\Delta|
\le
\varepsilon^2\alpha^2
\sum_{\substack{\gamma\in\widehat W\\ \gamma\text{ non-trivial}}}
\mathop{\mathbb E}\limits_{z\in\mathbb F_p^n/W} a_\gamma(z)^2
\le
\varepsilon^2\alpha^3.
$$
This proves the proposition.
\end{proof}

\section{Strengthened Chang's lemma for finite abelian groups}

We now examine how our strengthened version of Chang's lemma (Theorem \ref{theo:main}) can be extended to apply to finite abelian groups. Before doing so, we show how the ideas in the proof of Theorem~\ref{theo:main} lead to a short and self-contained proof of the classical Chang's lemma for finite abelian groups. This also serves to show how the iterative procedure adapts to finite abelian groups. 

Let $G$ be a finite abelian group, written additively, and let $\widehat G$ denote its dual group (also written additively). We assume that $0\in\widehat G$ denotes the trivial character. For any non-empty subset $S\subseteq G$, let $\mu_S$ denote the uniform probability measure on $S$. In general finite abelian groups one no longer has the linear-algebraic structure
available in $\mathbb F_p^n$. The right analogue is provided by dissociated sets
and their $\{-1,0,1\}$-spans. 

Given a finite set $\Lambda\subseteq \widehat G$, we write
$$
\langle \Lambda\rangle
=
\left\{
\sum_{\lambda\in\Lambda}\eta_\lambda\lambda
\;:\;
\eta_\lambda\in\{-1,0,1\}
\right\}
$$
for the $\{-1,0,1\}$-span of $\Lambda$.

We say that a finite set $\Lambda\subseteq \widehat G$ is dissociated if the only
relation of the form
$$
\sum_{\lambda\in\Lambda}\eta_\lambda\lambda=0,
\qquad
\eta_\lambda\in\{-1,0,1\},
$$
is the trivial one in which all coefficients $\eta_\lambda$ vanish.

For a function $f:G\to\mathbb C$ and a character $\xi\in\widehat G$, we define
$$
\widehat f(\xi)
=
\mathop{\mathbb E}\limits_{x\in G} f(x)\overline{\xi(x)}.
$$

Let $A\subseteq G$ be the set appearing in the statement of Chang's lemma, let
$\alpha$ be its density, and write $\mu=\mu_A$.
With this notation in place, Chang's lemma takes the following form for general abelian groups.

\begin{thm}[Chang's lemma for finite abelian groups]\label{Chang-general}
Let $A\subseteq G$ be a set of density $\alpha>0$, and let
$\varepsilon\in(0,1)$. Then there exists a dissociated set
$\Lambda\subseteq\widehat G$ with
$$
|\Lambda|\le \frac{2}{\varepsilon^2}\log\frac{1}{\alpha}
$$
such that, for every $\xi\notin\langle\Lambda\rangle$,
\begin{equation}\label{eqchang-general}
\left|
\mathop{\mathbb E}\limits_{x\sim\mu}\overline{\xi(x)}
\right|
\le \varepsilon.
\end{equation}
\end{thm}
\begin{proof}
We construct inductively a nested sequence of dissociated sets
$$
\Lambda_0\subseteq \Lambda_1\subseteq \cdots \subseteq \Lambda_T
$$
in $\widehat G$, and define for every $0\le i\le T$ a probability measure
$\nu_i$ on $G$, absolutely continuous with respect to $\mu_G$, whose density
has Fourier transform supported in $\langle \Lambda_i\rangle$. We let
$\Lambda_0=\varnothing$ and $\nu_0=\mu_G$, and assume that $\Lambda_i$ and
$\nu_i$ have already been defined.

If, for every $\xi\notin \langle\Lambda_i\rangle$, the inequality
\eqref{eqchang-general} is satisfied, then the process stops. Otherwise,
choose a character $\xi\notin\langle\Lambda_i\rangle$ witnessing the failure of
the conclusion, that is,
$$
\left|
\mathop{\mathbb E}\limits_{x\sim\mu}\overline{\xi(x)}
\right|
>
\varepsilon.
$$
We choose a phase factor $\theta=\theta(\xi)$ so as to align
$\mathop{\mathbb E}\limits_{x\sim\mu}\overline{\xi(x)}$ with the positive real
axis. If this quantity is non-zero, we define $\theta$ to be its complex
conjugate normalized to have modulus $1$; otherwise we let $\theta=1$. We use
this phase factor to define a real-valued test function $d:G\to[-1,1]$ by
$$
d(x)=\Re\bigl(\theta\,\overline{\xi(x)}\bigr).
$$
With this definition,
$$
\mathop{\mathbb E}\limits_{x\sim\mu}d(x)
=
\Re\left(
\theta\,
\mathop{\mathbb E}\limits_{x\sim\mu}\overline{\xi(x)}
\right)
=
\left|
\mathop{\mathbb E}\limits_{x\sim\mu}\overline{\xi(x)}
\right|
>
\varepsilon.
$$
Set $\Lambda_{i+1}=\Lambda_i\cup\{\xi\}$. Since
$\xi\notin\langle\Lambda_i\rangle$, the set $\Lambda_{i+1}$ is again
dissociated. Moreover,
$$
1+\varepsilon d(x)
=
1+\frac{\varepsilon}{2}\theta\,\overline{\xi(x)}
+\frac{\varepsilon}{2}\overline{\theta}\,\xi(x),
$$
so the Fourier transform of $x\mapsto 1+\varepsilon d(x)$ is supported on
$\{0,\xi,-\xi\}$. Writing $\nu_i=f_i\mu_G$, where
$f_i:G\to[0,\infty)$ is the density of $\nu_i$ with respect to $\mu_G$, and
recalling that $\widehat f_i$ is supported in $\langle\Lambda_i\rangle$, it
follows that the Fourier transform of $x\mapsto f_i(x)(1+\varepsilon d(x))$
is supported in
$$
\langle\Lambda_i\rangle+\{0,\xi,-\xi\}
\subseteq
\langle\Lambda_{i+1}\rangle.
$$
We now define
$$
\nu_{i+1}(x)=(1+\varepsilon d(x))\nu_i(x).
$$
Then $\nu_{i+1}$ is non-negative, and its density with respect to $\mu_G$ has
Fourier transform supported in $\langle\Lambda_{i+1}\rangle$. We next check that $\nu_{i+1}$ is still a probability measure. Since
$\nu_i=f_i\mu_G$, we have
$$
\mathop{\mathbb E}\limits_{x\sim\nu_i}d(x)
=
\Re\left(
\theta\,
\mathop{\mathbb E}\limits_{x\sim\nu_i}\overline{\xi(x)}
\right)
=
\Re\left(
\theta\,
\mathop{\mathbb E}\limits_{x\sim\mu_G}f_i(x)\overline{\xi(x)}
\right)
=
\Re\bigl(\theta\,\widehat f_i(\xi)\bigr).
$$
Since $\widehat f_i$ is supported in $\langle\Lambda_i\rangle$ and
$\xi\notin\langle\Lambda_i\rangle$, we have $\widehat f_i(\xi)=0$. Hence
$
\mathop{\mathbb E}\limits_{x\sim\nu_i}d(x)=0.
$
Therefore
$$
\sum_{x\in G}\nu_{i+1}(x)
=
\mathop{\mathbb E}\limits_{x\sim\nu_i}(1+\varepsilon d(x))
=
1+\varepsilon\mathop{\mathbb E}\limits_{x\sim\nu_i}d(x)
=
1,
$$
so $\nu_{i+1}$ is a probability measure. Define the potential $\Psi_i=D(\mu\|\nu_i)$. Since
$$
\log\frac{\mu(x)}{\nu_{i+1}(x)}
=
\log\frac{\mu(x)}{\nu_i(x)}
-
\log(1+\varepsilon d(x)),
$$
averaging with respect to $\mu$ gives
$
\Psi_{i+1}
=
\Psi_i-
\mathop{\mathbb E}\limits_{x\sim\mu}\log(1+\varepsilon d(x)).
$
We now lower bound the last expectation. Since $t\mapsto \log(1+\varepsilon t)$
is concave on $[-1,1]$, its graph lies above the chord joining the endpoints.
Thus, for every $t\in[-1,1]$,
$$
\log(1+\varepsilon t)
\ge
\frac{1+t}{2}\log(1+\varepsilon)
+
\frac{1-t}{2}\log(1-\varepsilon).
$$
Averaging with respect to $\mu$ gives
$$
\mathop{\mathbb E}\limits_{x\sim\mu}\log(1+\varepsilon d(x))
\ge
\frac{1+\mathop{\mathbb E}\limits_{x\sim\mu}d(x)}{2}\log(1+\varepsilon)
+
\frac{1-\mathop{\mathbb E}\limits_{x\sim\mu}d(x)}{2}\log(1-\varepsilon).
$$
The right-hand side is increasing as a function of
$\mathop{\mathbb E}\limits_{x\sim\mu}d(x)$. Since
$\mathop{\mathbb E}\limits_{x\sim\mu}d(x)>\varepsilon$, we obtain
$$
\mathop{\mathbb E}\limits_{x\sim\mu}\log(1+\varepsilon d(x))
\ge
\frac{1+\varepsilon}{2}\log(1+\varepsilon)
+
\frac{1-\varepsilon}{2}\log(1-\varepsilon).
$$
Finally, the function
$$
h(t)=
\frac{1+t}{2}\log(1+t)
+
\frac{1-t}{2}\log(1-t)
$$
satisfies $h(0)=0$ and $h'(t)\ge t
$
for every $t\in[0,1)$. Therefore $h(\varepsilon)\ge \varepsilon^2/2$, and hence
$$
\Psi_{i+1}
\le
\Psi_i-\frac{\varepsilon^2}{2}.
$$
In the end, since $\nu_0=\mu_G$ and $\mu=\mu_A$, we have
$$
\Psi_0=D(\mu_A\|\mu_G)=\log\frac{1}{\alpha}.
$$
On the other hand, $\Psi_i\ge 0$ for every $i$. Therefore the process must
terminate after at most
$$
T\le \frac{2}{\varepsilon^2}\log\frac{1}{\alpha}
$$
steps, and identifying $\Lambda$ with $\Lambda_T$ concludes the proof.
\end{proof}
Before passing to the refined case, let us isolate the form of the finite-field
statement which has a meaningful analogue for arbitrary finite abelian groups.
If $\mu_A$ denotes the uniform probability measure on $A$, then the conclusion
of Theorem~\ref{theo:main} can be rewritten as follows. For every
$\xi\notin V$,
$$
\mathop{\mathbb E}\limits_{y\in \mathbb F_p^n/V^\perp}
\left|
|\mathbb F_p^n/V^\perp|\,
\mathop{\mathbb E}\limits_{x\sim\mu_A}
\mathbf 1_{y+V^\perp}(x)\,\overline{\chi_\xi(x)}
\right|
\le \varepsilon.
$$
For each coset $y+V^\perp$, the function
$$
x\mapsto |\mathbb F_p^n/V^\perp|\,\mathbf 1_{y+V^\perp}(x)
$$
is non-negative, has global average $1$, and has Fourier support contained in
$V$. Moreover these functions average to $1$ pointwise as $y$ ranges over
$\mathbb F_p^n/V^\perp$. Thus the refined finite-field statement may be read
as follows: after removing a low-dimensional exceptional space $V$, every
remaining character has small average against $\mu_A$ even after localizing
$\mu_A$ by the elements of a positive Fourier-supported decomposition of
unity.

This formulation is the one that survives in general finite abelian groups. The role of the
coset densities is played by a family of non-negative functions
$g_\sigma$ with mean $1$, Fourier support contained in the exceptional span,
and pointwise average equal to $1$. The refined statement below says that,
outside this span, one still has cancellation after localization by this whole
family of weights.

For each integer $r\ge 0$, we write
$
\Sigma_r=\{\pm1\}^r,
$
with the convention that $\Sigma_0=\{\varnothing\}$. We can now state the
result.
\begin{thm}[Refined Chang's lemma for finite abelian groups]
\label{thm:refined-chang-general}
Let $G$ be a finite abelian group, let $A\subseteq G$ be a set of density
$\alpha>0$, and let $\varepsilon\in(0,1)$. Then there exist an integer
$r\ge 0$, a dissociated set $\Delta\subseteq \widehat G$ with $|\Delta|=r$,
and functions $g_\sigma:G\to[0,\infty)$, indexed by $\sigma\in\Sigma_r$,
with
$$
r\le \frac{2}{\varepsilon^2}\log\frac1\alpha.
$$
These functions may be chosen so that, for every $\sigma\in\Sigma_r$,
$$
\mathbb E_{x\in G} g_\sigma(x)=1
\qquad\text{and}\qquad
\operatorname{supp}\widehat{g_\sigma}\subseteq \langle\Delta\rangle,
$$
while, for every $x\in G$,
$
\mathbb E_{\sigma\in\Sigma_r} g_\sigma(x)=1.
$
Moreover, for every $\eta\notin\langle\Delta\rangle$, one has
$$
\mathbb E_{\sigma\in\Sigma_r}
\left|
\mathbb E_{x\sim\mu_A} g_\sigma(x)\,\overline{\eta(x)}
\right|
\le \varepsilon.
$$
\end{thm}
\begin{proof}
We construct inductively dissociated sets
$$
\Delta_0\subseteq \Delta_1\subseteq \cdots
$$
and, for each $i\ge 0$, functions $g_\sigma:G\to[0,\infty)$, indexed by
$\sigma\in\Sigma_i$, such that
$$
\mathbb E_{x\in G}g_\sigma(x)=1,
\qquad
\operatorname{supp}\widehat{g_\sigma}\subseteq \langle\Delta_i\rangle
$$
for every $\sigma\in\Sigma_i$, and
$\mathbb E_{\sigma\in\Sigma_i}g_\sigma(x)=1$ for every $x\in G$. We start with
$\Delta_0=\varnothing$ and $g_\varnothing=1$.

Assume that the construction has reached step $i$. Set
$m_\sigma=\mathbb E_{x\sim\mu_A}g_\sigma(x)$ and define a probability measure
$p_i$ on $\Sigma_i$ by $p_i(\sigma)=2^{-i}m_\sigma$. This is indeed a
probability measure, since
$$
\sum_{\sigma\in\Sigma_i}p_i(\sigma)
=
\mathbb E_{x\sim\mu_A}\mathbb E_{\sigma\in\Sigma_i}g_\sigma(x)
=
1.
$$
Let $u_i$ be the uniform probability measure on $\Sigma_i$, and set
$\Phi_i=D(p_i\|u_i)$. Since, for each fixed $x\in G$, the numbers
$2^{-i}g_\sigma(x)$ form a probability distribution on $\Sigma_i$, and since
$p_i(\sigma)=\sum_x\mu_A(x)2^{-i}g_\sigma(x)$ while
$u_i(\sigma)=\sum_x\mu_G(x)2^{-i}g_\sigma(x)$, the log-sum inequality gives
$$
0\le \Phi_i\le D(\mu_A\|\mu_G)=\log\frac1\alpha.
$$
If, for every $\eta\notin\langle\Delta_i\rangle$, one has
$$
\mathbb E_{\sigma\in\Sigma_i}
\left|
\mathbb E_{x\sim\mu_A}g_\sigma(x)\overline{\eta(x)}
\right|
\le \varepsilon,
$$
then the construction stops. Otherwise choose $\eta\notin\langle\Delta_i\rangle$
such that
$$
\mathbb E_{\sigma\in\Sigma_i}
\left|
\mathbb E_{x\sim\mu_A}g_\sigma(x)\overline{\eta(x)}
\right|
>
\varepsilon.
$$
For each $\sigma\in\Sigma_i$, choose $|\theta_\sigma|=1$ so that
$$
\Re\left(
\theta_\sigma
\mathbb E_{x\sim\mu_A}g_\sigma(x)\overline{\eta(x)}
\right)
=
\left|
\mathbb E_{x\sim\mu_A}g_\sigma(x)\overline{\eta(x)}
\right|,
$$
and put
$
d_\sigma(x)=\Re(\theta_\sigma\overline{\eta(x)}).
$
Equivalently,
$
d_\sigma(x)
=
\frac12\theta_\sigma\overline{\eta(x)}
+
\frac12\overline{\theta_\sigma}\eta(x).
$
Then $|d_\sigma|\le 1$, and, since
$\operatorname{supp}\widehat{g_\sigma}\subseteq\langle\Delta_i\rangle$ while
$\eta,-\eta\notin\langle\Delta_i\rangle$, one has
$$
\mathbb E_{x\in G}g_\sigma(x)d_\sigma(x)
=
\frac12\theta_\sigma\widehat{g_\sigma}(\eta)
+
\frac12\overline{\theta_\sigma}\widehat{g_\sigma}(-\eta)
=
0.
$$
We split each $g_\sigma$ into
$$
g_{\sigma,+}=(1+d_\sigma)g_\sigma,
\qquad
g_{\sigma,-}=(1-d_\sigma)g_\sigma.
$$
These functions are non-negative and have mean $1$. Moreover,
$(g_{\sigma,+}(x)+g_{\sigma,-}(x))/2=g_\sigma(x)$, so, after identifying
$\Sigma_{i+1}$ with $\Sigma_i\times\{\pm1\}$, we still have
$\mathbb E_{\tau\in\Sigma_{i+1}}g_\tau(x)=1$ for every $x\in G$.

Set $\Delta_{i+1}=\Delta_i\cup\{\eta\}$. Since
$\eta\notin\langle\Delta_i\rangle$, the set $\Delta_{i+1}$ is still
dissociated. Also, the identity above shows that the Fourier transform of
$1\pm d_\sigma$ is supported on $\{0,\eta,-\eta\}$. Therefore
$$
\operatorname{supp}\widehat{g_{\sigma,\pm}}
\subseteq
\operatorname{supp}\widehat{g_\sigma}+\{0,\eta,-\eta\}
\subseteq
\langle\Delta_{i+1}\rangle.
$$
Thus the inductive properties are preserved.

It remains to check that the entropy potential increases by a definite amount
at every non-terminal step. Put
$
c_\sigma=
\left|
\mathbb E_{x\sim\mu_A}g_\sigma(x)\overline{\eta(x)}
\right|.
$
Then $0\le c_\sigma\le m_\sigma$, and, by the choice of $\theta_\sigma$,
$\mathbb E_{x\sim\mu_A}g_\sigma(x)d_\sigma(x)=c_\sigma$. Hence
$\mathbb E_{x\sim\mu_A}g_{\sigma,\pm}(x)=m_\sigma\pm c_\sigma$, and therefore
$$
p_{i+1}(\sigma,\pm)
=
\frac12 p_i(\sigma)
\left(1\pm\frac{c_\sigma}{m_\sigma}\right),
$$
with the convention that $c_\sigma/m_\sigma=0$ when $m_\sigma=0$. A direct
calculation gives
$$
\Phi_{i+1}-\Phi_i
=
\mathbb E_{\sigma\sim p_i}
\left[
\frac{1+c_\sigma/m_\sigma}{2}
\log\left(1+\frac{c_\sigma}{m_\sigma}\right)
+
\frac{1-c_\sigma/m_\sigma}{2}
\log\left(1-\frac{c_\sigma}{m_\sigma}\right)
\right].
$$
The function
$$
t\mapsto
\frac{1+t}{2}\log(1+t)
+
\frac{1-t}{2}\log(1-t)
$$
is convex, and at least $t^2/2$ on $[0,1]$. Jensen's
inequality therefore gives
$$
\Phi_{i+1}-\Phi_i
\ge
\frac12
\left(
\mathbb E_{\sigma\sim p_i}\frac{c_\sigma}{m_\sigma}
\right)^2.
$$
Since $p_i(\sigma)=2^{-i}m_\sigma$, we have
$\mathbb E_{\sigma\sim p_i}c_\sigma/m_\sigma
=
\mathbb E_{\sigma\in\Sigma_i}c_\sigma>\varepsilon$. Thus every non-terminal
step gives $\Phi_{i+1}-\Phi_i>\varepsilon^2/2$.

Since $0\le \Phi_i\le \log(1/\alpha)$ at every stage, the construction must
stop after at most
$
\frac{2}{\varepsilon^2}\log\frac1\alpha
$
steps. If $T$ is the first terminal step, set $r=T$ and
$\Delta=\Delta_T$. The terminal condition gives
$$
\mathbb E_{\sigma\in\Sigma_r}
\left|
\mathbb E_{x\sim\mu_A}g_\sigma(x)\overline{\eta(x)}
\right|
\le \varepsilon
\qquad\text{for every }\eta\notin\langle\Delta\rangle,
$$
and all the remaining properties of the functions $g_\sigma$ were preserved
throughout the induction.
\end{proof}
Note that Theorem~\ref{thm:refined-chang-general} recovers Chang's lemma in
the form of Theorem~\ref{Chang-general}. Indeed, for every
$\eta\notin\langle\Delta\rangle$,
$$
\left|\mathbb E_{x\sim\mu_A}\overline{\eta(x)}\right|
=
\left|
\mathbb E_{\sigma\in\Sigma_r}
\mathbb E_{x\sim\mu_A}g_\sigma(x)\overline{\eta(x)}
\right|
\le
\mathbb E_{\sigma\in\Sigma_r}
\left|
\mathbb E_{x\sim\mu_A}g_\sigma(x)\overline{\eta(x)}
\right|
\le \varepsilon,
$$
because $\mathbb E_{\sigma\in\Sigma_r}g_\sigma(x)=1$ for every $x\in G$.

The formulation of Theorem~\ref{thm:refined-chang-general} is less ``geometric''
than the finite-field statement. Indeed, in $\mathbb F_p^n$, the functions appearing in
the refined theorem are simply the normalized indicators of the cosets of
$V^\perp$; in a general finite abelian group, the weights $g_\sigma$ do not necessarily admit such a transparent geometric interpretation. The most natural candidates to replace cosets in this general setting—most notably Bohr sets—fail to remain stable under the local refinements required here, at least if one aims to retain the sharp Chang bound $2\varepsilon^{-2}\log(1/\alpha)$.

Another related issue is that, in its current formulation, the result does not appear to admit an analogue of Proposition~\ref{prop:localized-counting}. 

\printbibliography

@article{Chang2002,
  author  = {Mei-Chu Chang},
  title   = {A polynomial bound in {Freiman's} theorem},
  journal = {Duke Mathematical Journal},
  volume  = {113},
  number  = {3},
  pages   = {399--419},
  year    = {2002},
  doi     = {10.1215/S0012-7094-02-11331-3}
}

@article{sanders2007littlewood,
  author  = {Tom Sanders},
  title   = {The {Littlewood--Gowers} problem},
  journal = {Journal d'Analyse Math{\'e}matique},
  volume  = {101},
  number  = {1},
  pages   = {123--162},
  year    = {2007},
  doi     = {10.1007/s11854-007-0005-1}
}

@article{sanders2011roth,
  author  = {Tom Sanders},
  title   = {On {Roth's} theorem on progressions},
  journal = {Annals of Mathematics},
  volume  = {174},
  number  = {1},
  pages   = {619--636},
  year    = {2011},
  doi     = {10.4007/annals.2011.174.1.20}
}

@inproceedings{tsang2013fourier,
  title={Fourier sparsity, spectral norm, and the log-rank conjecture},
  author={Tsang, Hing Yin and Wong, Chung Hoi and Xie, Ning and Zhang, Shengyu},
  booktitle={2013 IEEE 54th Annual Symposium on Foundations of Computer Science},
  pages={658--667},
  year={2013},
  organization={IEEE}
}

@article{helfgott2015growth,
  title={Growth in groups: ideas and perspectives},
  author={Helfgott, Harald},
  journal={Bulletin of the American Mathematical Society},
  volume={52},
  number={3},
  pages={357--413},
  year={2015}
}

@article{bloom2016quantitative,
  title={A quantitative improvement for Roth's theorem on arithmetic progressions},
  author={Bloom, Thomas F},
  journal={Journal of the London Mathematical Society},
  volume={93},
  number={3},
  pages={643--663},
  year={2016},
  publisher={Oxford University Press}
}

@article{bloom2020breaking,
  title={Breaking the logarithmic barrier in Roth's theorem on arithmetic progressions},
  author={Bloom, Thomas F and Sisask, Olof},
  journal={arXiv preprint arXiv:2007.03528},
  year={2020}
}

@inproceedings{asadi2022worst,
  title={Worst-case to average-case reductions via additive combinatorics},
  author={Asadi, Vahid R and Golovnev, Alexander and Gur, Tom and Shinkar, Igor},
  booktitle={Proceedings of the 54th Annual ACM SIGACT Symposium on Theory of Computing},
  pages={1566--1574},
  year={2022}
}

@inproceedings{lee2019fourier,
  title={Fourier Bounds and Pseudorandom Generators for Product Tests},
  author={Lee, Chin Ho},
  booktitle={34th Computational Complexity Conference (CCC 2019)},
  pages={7--1},
  year={2019},
  organization={Schloss Dagstuhl--Leibniz-Zentrum f{\"u}r Informatik}
}

@article{green2008quantitative,
  author  = {Ben Green and Tom Sanders},
  title   = {A quantitative version of the idempotent theorem in harmonic analysis},
  journal = {Annals of Mathematics},
  volume  = {168},
  number  = {3},
  pages   = {1025--1054},
  year    = {2008},
  doi     = {10.4007/annals.2008.168.1025}
}

@article{green2008boolean,
  author  = {Ben Green and Tom Sanders},
  title   = {Boolean functions with small spectral norm},
  journal = {Geometric and Functional Analysis},
  volume  = {18},
  number  = {1},
  pages   = {144--162},
  year    = {2008},
  doi     = {10.1007/s00039-008-0654-y}
}

@article{Breuillard2012,
  author  = {Emmanuel Breuillard and Ben Green and Terence Tao},
  title   = {The structure of approximate groups},
  journal = {Publications Math{\'e}matiques de l'IH{\'E}S},
  volume  = {116},
  pages   = {115--221},
  year    = {2012},
  doi     = {10.1007/s10240-012-0043-9}
}

@article{bloom2022new,
  author  = {Thomas F. Bloom and James Maynard},
  title   = {A new upper bound for sets with no square differences},
  journal = {Compositio Mathematica},
  volume  = {158},
  number  = {8},
  pages   = {1777--1798},
  year    = {2022},
  doi     = {10.1112/S0010437X22007679}
}

@article{chan2016approximate,
  author  = {Siu On Chan and James R. Lee and Prasad Raghavendra and David Steurer},
  title   = {Approximate constraint satisfaction requires large {LP} relaxations},
  journal = {Journal of the ACM},
  volume  = {63},
  number  = {4},
  pages   = {34:1--34:22},
  year    = {2016},
  doi     = {10.1145/2811255}
}

@article{Green2003,
  author  = {Ben Green},
  title   = {Some constructions in the inverse spectral theory of cyclic groups},
  journal = {Combinatorics, Probability and Computing},
  volume  = {12},
  number  = {2},
  pages   = {127--138},
  year    = {2003},
  doi     = {10.1017/S0963548302005436}
}

@article{impagliazzo2014entropic,
  author  = {Russell Impagliazzo and Cristopher Moore and Alexander Russell},
  title   = {An entropic proof of {Chang's} inequality},
  journal = {SIAM Journal on Discrete Mathematics},
  volume  = {28},
  number  = {1},
  pages   = {173--176},
  year    = {2014},
  doi     = {10.1137/120877982}
}

@misc{WigdersonAlmostPeriodicity,
  author   = {Yuval Wigderson},
  title    = {An Exposition of Almost Periodicity},
  year     = {2023},
  note     = {Expository notes, unpublished},
  url      = {https://ywigderson.math.ethz.ch/math/static/AlmostPeriodicity.pdf},
  urldate  = {2026-03-15}
}

@article{Lee2017LargeSpectrum,
  title     = {Covering the large spectrum and generalized {R}iesz products},
  author    = {Lee, James R.},
  journal   = {SIAM Journal on Discrete Mathematics},
  volume    = {31},
  number    = {1},
  pages     = {562--572},
  year      = {2017},
  publisher = {Society for Industrial and Applied Mathematics},
  doi       = {10.1137/15M1048604}
}

@inproceedings{HebertJohnsonKimReingoldRothblum2018,
  author    = {Ursula H{\'e}bert-Johnson and Michael Kim and Omer Reingold and Guy Rothblum},
  title     = {Multicalibration: Calibration for the (Computationally-Identifiable) Masses},
  booktitle = {Proceedings of the 35th International Conference on Machine Learning},
  series    = {Proceedings of Machine Learning Research},
  volume    = {80},
  pages     = {1939--1948},
  publisher = {PMLR},
  year      = {2018}
}
\end{document}